\documentclass[12pt]{amsart}

\makeatletter
\@namedef{subjclassname@2020}{\textup{2020} Mathematics Subject Classification}
\makeatother
\newcommand{\sodd}{\sig_{\text{odd}}}

\newcommand{\mg}{\infty}

\newcommand{\pii}{\pi i}

\newcommand{\R}{\mathcal{R}}

\newcommand{\dsum}{\di\sum}

\DeclareFontFamily{U}{mathx}{\hyphenchar\font45}
\DeclareFontShape{U}{mathx}{m}{n}{
      <5> <6> <7> <8> <9> <10>
      <10.95> <12> <14.4> <17.28> <20.74> <24.88>
      mathx10
      }{}
\DeclareSymbolFont{mathx}{U}{mathx}{m}{n}
\DeclareFontSubstitution{U}{mathx}{m}{n}
\DeclareMathAccent{\widecheck}{0}{mathx}{"71}
\newcommand{\inrr}{\ensuremath{\in\rr}}

\renewcommand{\kill}[1]{}
\newcommand{\dummy}[1]{\mbox{}}

\makeatletter
\newcommand{\xequal}[2][]{\ext@arrow 0055{\equalfill@}{#1}{#2}}
\def\equalfill@{\arrowfill@\Relbar\Relbar\Relbar}
\makeatother

\newcommand{\mto}{\mapsto}

\newcommand{\1}{\ensuremath{\ol{\mathrm{P}}}}

\renewcommand{\k}{\ensuremath{\ol{\mathrm{P}}}}

\renewcommand{\k}[1]{\ensuremath{\left({#1}\right)}}

\newcommand{\bca}{\begin{cases}}
\newcommand{\eca}{\end{cases}}

\newcommand{\mug}{\ensuremath{\infty}}

\newcommand{\dprod}{\di\prod}

\newcommand{\ff}[2]{\ensuremath{\di\fr{#1}{#2}}}

\newcommand{\s}[1]{\ensuremath{\di\int{#1}\,dx}}

\newcommand{\bpic}{\begin{picture}}\newcommand{\epic}{\end{picture}}

\newcommand{\beda}{\begin{edaenumerate}}
\newcommand{\eeda}{\end{edaenumerate}}

\newcommand{\cd}{\cdots}

\newcommand{\sh}[1]{\shadowbox{#1}}
\newcommand{\st}{\strut}

\newcommand{\q}{\quad}

\newcommand{\bq}{\begin{quote}}\newcommand{\eq}{\end{quote}}

\newcommand{\rt}{\sqrt}
\newcommand{\be}{\begin{enumerate}}\newcommand{\ee}{\end{enumerate}}
\newcommand{\bce}{\begin{center}}\newcommand{\ece}{\end{center}}
\newcommand{\bde}{\begin{description}}\newcommand{\ede}{\end{description}}
\newcommand{\bri}{\begin{flushright}}\newcommand{\eri}{\end{flushright}}
\newcommand{\bb}{\begin{block}}\newcommand{\eb}{\end{block}}
\newcommand{\bt}{\begin{thm}}\newcommand{\et}{\end{thm}}
\newcommand{\bpf}{\begin{proof}}\newcommand{\epf}{\end{proof}}
\newcommand{\bex}{\begin{ex}}\newcommand{\eex}{\end{ex}}
\newcommand{\bexr}{\begin{exr}}\newcommand{\eexr}{\end{exr}}
\newcommand{\bft}{\begin{fact}}\newcommand{\eft}{\end{fact}}
\newcommand{\brk}{\begin{rmk}}\newcommand{\erk}{\end{rmk}}
\newcommand{\ba}{\begin{align*}}\newcommand{\ea}{\end{align*}}
\newcommand{\bexe}{\begin{exe}}\newcommand{\eexe}{\end{exe}}
\newcommand{\tn}{\textnormal}

\newcommand{\bit}{\begin{itemize}}\newcommand{\eit}{\end{itemize}}

\newcommand{\bcm}{\begin{comment}}\newcommand{\ecm}{\end{comment}}
\newcommand{\ol}{\overline}\newcommand{\ul}{\underline}
\newcommand{\hf}{\hfill}
\newcommand{\fr}{\frac}
\newcommand{\cc}{\ensuremath{\mathbf{C}}}
\newcommand{\nn}{\ensuremath{\mathbf{N}}}

\newcommand{\rr}{\ensuremath{\mathbf{R}}}
\newcommand{\zz}{\ensuremath{\mathbf{Z}}}

\newcommand{\bd}{\begin{defn}}\newcommand{\ed}{\end{defn}}
\newcommand{\bp}{\begin{prop}}\newcommand{\ep}{\end{prop}}
\newcommand{\p}{\ensuremath{\pi}}
\newcommand{\eh}{\emph}
\newcommand{\sub}{\subseteq}

\newcommand{\mb}{\mbox}
\newcommand{\ph}{\phantom}

\newcommand{\di}{\displaystyle}

\renewcommand{\d}{\ensuremath{\bm{d}}}

\newcommand{\f}{\frac}

\newcommand{\y}{\ensuremath{\bm{y}}}
\newcommand{\z}{\ensuremath{\bm{z}}}
\newcommand{\np}{\newpage}

\renewcommand{\y}{\mathbf{y}}

\renewcommand{\s}{\sigma}

\renewcommand{\d}{\delta}
\renewcommand{\P}{\mathcal{P}}
\newcommand{\Q}{\mathcal{Q}}

\renewcommand{\d}{\delta}

\usepackage[dvipdfmx]{graphicx}
\usepackage[dvipsnames]{xcolor}
\usepackage{float,afterpage}
\usepackage{asymptote,layout}
\usepackage{wrapfig,epic}
\usepackage{amsmath,amssymb,cases,pict2e}

\usepackage{multicol}
\usepackage{tikz}

\usepackage{geometry}%morisawa,okumacro,plext,
\usepackage{exscale,latexsym,bm}%,ascmac
\usepackage{amssymb,enumerate,amsmath,amsthm,amsfonts}
\usepackage{verbatim,fancybox,wasysym,fancyhdr,type1cm}
\usepackage[frame,all,poly,curve,knot,arrow]{xy}
\usepackage{colortbl}
\usepackage{boxedminipage}
\usepackage{multirow}%\usepackage{mathabx}
\graphicspath{{./figs/}}
\theoremstyle{definition}
\newtheorem{thm}{Theorem}[section]

\newtheorem{lem}[thm]{Lemma}
\newtheorem{prop}[thm]{Proposition}\newtheorem{cor}[thm]{Corollary}

\newtheorem{exr}[thm]{Exercise}

\newtheorem{ex}[thm]{Example}

\newtheorem{defn}[thm]{Definition}\newtheorem{rmk}[thm]{Remark}
\newtheorem{fact}[thm]{Fact}
\newtheorem{block}[thm]{}
\newtheorem*{exe}{Exercise}

\renewcommand{\P}{\mathbf{P}}
\renewcommand{\z}{\zeta}

\newcommand{\tf}{\tfrac}
\newcommand{\tff}{\tfrac}

\newcommand{\ps}{\psi}
\begin{document}%:\UTF{2022}%

\title{Applications of Ramanujan's work on Eisenstein series}
%signed bigrassmannian polynomials}
\author{Masato Kobayashi}
%\author{Shunji Sasaki}
%\thanks{corresponding author:Masato Kobayashi, masato210@gmail.com}
\date{\today}                                       % Activate to display a given date
% or no date
%\thanks{corresponding author}
%\dedicatory{Dedicated to }
%\subjclass[2020]{Primary:11M32}

%Jacobi triple product identity, pentagonal numbers, triangular numbers. 
%}
\address{Masato Kobayashi\\
Department of Engineering\\
Kanagawa University, 3-27-1 Rokkaku-bashi, Yokohama 221-8686, Japan.}
\email{masato210@gmail.com}
%\address{Shunji Sasaki\\
%Kawaguchi public Kamiaoki junior high school\\
%3-9-1 Kamiaoki-Nishi, Kawaguchi 333-0845, Japan.
%}

%\address{Graduate School of Science and Engineering\\
%Department of Mathematics\\
%Saitama University,
%255 Shimo-Okubo, Saitama 338-8570, Japan.}
%\email{schnittkejp@me.com}
%\curraddr{}
%\address{Department of Mathematics\\
%the University of Tennessee, Knoxville, TN 37996}

\maketitle
\begin{abstract}
Ramanujan (1916) expressed quotients of certain $q$-series as polynomials of the Eisenstein series $P, Q, R$ and derived the famous Ramanujan's differential equations. We continue this research with the variants of Eisenstein-type series which Hahn (2007) recently introduced. We also prove new formulas of convolution sums for divisor sum functions as subsequent work of Cheng-Williams (2004) and 
Huard-Ou-Spearman-Williams (2002).
\end{abstract}
\tableofcontents

\section{Introduction}%:\UTF{2022}%
\subsection{Eisenstein series}%:\UTF{2022}%
\eh{Eisenstein series} are one of the important topics in the number theory; 
in particular, they play a significant role in the theory of modular forms, Ramanujan's theta functions and elliptic functions. 
Let us make a definition of these series after preparing several words and symbols. For $n\in\nn$ and $s\inrr$, 
the \eh{divisor sum function} of weight $s$ is  
$\sigma_{s}(n)=\sum_{d|n}d^{s}$. 
Often, we write $\s_{1}(n)=\s(n)$. 
In the sequel, we always assume $q, x\in\cc$ with  $|q|<1$ and  $|x|<1$. Define signed Bernoulli numbers $(B_{n})_{n\ge0}$ by 
\[
\ff{x}{e^{x}-1}=
\dsum_{n=0}^{\mug}\ff{B_{n}}{n!}x^{n}.
\]
The first few values of nonzero $B_{n}$ are 
\[
B_{0}=1, B_{1}=-\ff{1}{2}, 
B_{2}=\ff{1}{6}, B_{4}=-\ff{1}{30}, 
B_{6}=\ff{1}{42}, B_{8}=-\ff{1}{30}, B_{10}=\ff{5}{66}.	
\]
%Notice that there is a relation to Lambert series.
%\[
%\dsum_{n=1}^{\mug}\ff{n^{s}q^{n}}{1-q^{n}}=
%\dsum_{n=1}^{\mug}\s_{s}(n)q^{n}.
%\]
\begin{defn}[Normalized Eisenstein series]
\[
E_{2k}(q)=
1-\ff{4k}{B_{2k}}
\dsum_{n=1}^{\mug}
\s_{2k-1}(n)q^{n},\q k\ge1.
\]
\end{defn}
We remark that it is more traditional to regard Eisenstein series as functions of $\tau\in \cc$, $\text{Im}(\tau)>0$ with $q=e^{2\pii\tau}$. 
Customarily, there is particular usage for the first three Eisenstein series as
\begin{align*}
	P(q)&=1-24
	\dsum_{n=1}^{\mug}\s(n)q^{n},
	\\Q(q)&=E_{4}(q)=1+240
	\dsum_{n=1}^{\mug}\s_{3}(n)q^{n},
	\\R(q)&=E_{6}(q)=1-504
	\dsum_{n=1}^{\mug}\s_{5}(n)q^{n}.
\end{align*}
%with understanding 
%$\s(0)=-\tf{1}{24}$,
%$\s_{3}(0)=\tf{1}{240}$,
%$\s_{5}(0)=-\tf{1}{504}$ so that 
%constant terms are always 1. 
In the literature, some authors also use symbols $L(q), M(q), N(q)$ for $P(q), Q(q), R(q)$, respectively. 
Whenever there is no confusion from the context, we abbreviate $P(q), Q(q), R(q)$ to $P, Q, R$.

\subsection{Main results}%:\UTF{2022}%

Our main results consist of three kinds of theorems in  Sections \ref{s2}, \ref{s3}, \ref{s4}, respectively. 
Let us describe some background and our motivation.

One of the classic topics in modular forms and 
$q$-series is to find an expansion of 
$(q;q)_{\mug}^{r} (r>0)$ where 
\[
(q;q)_{\mug}=\dprod_{n=1}^{\mug}(1-q^{n}). 
\]
There are beautiful results by Euler ($r=1$), Gauss-Jacobi $(r=3)$, Klein-Fricke ($r=8$), Dyson, Atkin,  Macdonald and many others.
More recently, Berndt-Chan-Liu-Yesilyurt (2004) \cite{bcly} found 
\begin{align*}
	32(q;q)_{\mg}^{10}&=9
\dsum_{m=-\mug}^{\mug}
(-1)^{m}(2m+1)^{3}q^{3m(m+1)/2}
\dsum_{n=-\mug}^{\mug}(-1)^{n}(2n+1)q^{n(n+1)/6}
	\\&\ph{=}-
\dsum_{m=-\mug}^{\mug}(-1)^{m}(2m+1)q^{3m(m+1)/2}
\dsum_{n=-\mug}^{\mug}(-1)^{n}(2n+1)^{3}
q^{n(n+1)/6}.
\end{align*}
Similarly, 
Chan-Cooper-Toh (2006) \cite{cct} proved 
\begin{align*}
	16308864q^{26/24}(q;q)_{\mg}^{26}&=\sum_{i, j=-\mg}^{\mg}(-1)^{i+j}
f\k{\ff{(6i+1)^{2}}{2}, \ff{(6j+1)^{2}}{2}}
q^{\tf{(6i+1)^{2}+(6j+1)^{2}}{24}}
	\\&\ph{=}+
\sum_{i, j=-\mg}^{\mg}(-1)^{i+j}
f\k{12i^{2}, (6j+1)^{2}}
q^{\tf{12i^{2}+(6j+1)^{2}}{12}}
\end{align*}
where 
\[
f(m, n)=
m^{6}-66m^{5}n+495m^{4}n^{2}-924m^{3}n^{3}+495m^{2}n^{4}-66mn^{5}+n^{6}.
\]
%Along these lines, Ramanujan's work \cite{ra}; 
In their discussions, Ramanujan's work (1916) \cite{ra} on Eisenstein series plays a fundamental role; we will give  more details in subsection \ref{s21}. 
Motivated by their results, we establish our main results in Section \ref{s2} as Theorems \ref{t3}, \ref{t4}, \ref{tr} and Corollary \ref{t5}.

%Hahn (2007) \cite{ha} found  further studied them and found many relations. Here we prove 
One typical application of Eisenstein series is 
to derive arithmetic identities on divisor sum functions. 
Historically, for example, Glaisher, Lahiri, MacMahon, Ramanujan discovered such formulas. Among those,  one simple formula comes from the relation $E_{8}(q)=E_{4}^{2}(q)$ as follows.
\begin{thm}\label{t9}
For each $n\ge0$, we have 
%\[
%\s_{3}(n)=
%\ff{1}{5}
%\k{(6n-1)\s(n)+
%12 \dsum_{k=1}^{n-1}\s(k)\s(n-k)}
%\]
\[
\s_{7}(n)=120
\dsum_{
\substack{i, j\ge0\\
i+j=n}
}\s_{3}(i)\s_{3}(j)
\]
with 
$\s_{3}(0)=\tf{1}{240}$ and $\s_{7}(0)=\tf{1}{480}.$ 
\end{thm}
More recent developments are due to Cheng-Williams (2004) \cite{cw}, Hahn (2007) \cite{ha}, 
 Huard-Ou-Spearman-Williams (2002) \cite{hu} and Melfi (1998) \cite{me} etc. For example, Cheng-Williams proved 
\[
\sum_{m\le n}\s(4m-3)\s(4n-(4m-3))
=4\s_{3}(n)-4\s_{3}(\tf{n}{2}).
\] 
The idea of both Cheng-Williams and Hahn is to use certain parametrization of Eisenstein series originating  from the theory of elliptic functions. Melfi makes use of modular forms while the method of Huard-Ou-Spearman-Williams is quite elementary. Inspired by all of  these work, in Section \ref{s3}, we show several new convolution sums on divisors as Theorems \ref{t10}, \ref{t11}, Corollary \ref{cor1} and 
Theorem \ref{t12}.

In Section \ref{s4}, we prove some miscellaneous results. Recently, Berndt-Reb\'{a}k (2021) \cite{br}, 
Paek-Shin-Yi (2018) \cite{psy} and Yi-Cho-Kim-Lee-Yu-Paek (2013) \cite{yck} and other researchers discussed special values of the Ramanujan theta functions $\varphi, \psi$ and cubic continued fractions $G$. 
Along the same line, we find special values of certain  series as Theorems \ref{t6}, \ref{t7}, \ref{t8}.

With these new results, we wish to contribute  development of our research in this area.

%\begin{thm}
%is a polynomial of 
%\end{thm}

%\np%----------------------------------------------------
\subsection{Notation}%:\UTF{2022}%
%Throughout this article, $i, j, k, n$ denote a nonnegative integer unless otherwise specified. 
Following Berndt's book \cite{be}, we collect some notation and symbols which frequently appear in the Ramanujan theory.

\begin{align*}
	(x;q)_{\mg}&=\dprod_{n=0}^{\mug} (1-xq^{n}).
	\\(q;q)_{\mug}&=\dprod_{n=1}^{\mug} (1-q^{n})
=
\dsum_{k=-\mg}^{\mug}(-1)^{k}
q^{k(3k-1)/{2}}.
	\\f(a, b)&=
	\dsum_{k=-\mug}^{\mug}a^{k(k+1)/2}b^{(k-1)k/2},
	\q |ab|<1.
	\\\varphi(q)&=f(q, q)=1+2 \dsum_{k=1}^{\mug}q^{k^{2}}.
	\\\psi(q)&=f(q, q^{3})
=
\dsum_{k=0}^{\mug}q^{k(k+1)/2}.
	\\f(-q)&=f(-q, -q^{2})\, (=(q;q)_{\mg}).
	\\y&=\p \ff{\mb{}_{2}F_{1}(\tf{1}{2}, \tf{1}{2};1;1-x)}
{\mb{}_{2}F_{1}(\tf{1}{2}, \tf{1}{2};1;x)}. 
	\\z&=\mb{}_{2}F_{1}(\tf{1}{2}, \tf{1}{2};1;x).
	\\q&=e^{-y}.
\end{align*}
Here $\mb{}_{2}F_{1}$ is the Gaussian hypergeometric  series.
\begin{fact}[{\cite[p.129, 127]{be}, \cite[p.356]{ab}}]
\label{f0}
Let $x, z, q$ be given as above. Then 
\begin{align*}
	P&=(1-5x)z^{2}+12x(1-x)z\ff{dz}{dx},
	\\Q&=z^{4}(1+14x+x^{2}),
	\\R&=z^{6}(1+x)(1-34x+x^{2}),
	\\qP'&=\ff{P^{2}-Q}{12},
	\\qQ'&=\ff{PQ-R}{3},
	\\qR'&=\ff{PR-Q^{2}}{2}
\end{align*}
where $\mb{}'$ denotes the derivative in $q$.
\end{fact}

%\begin{fact}
%$E_{2k}(q) (2k\ge8)$ is a polynomial of 
%$P, Q, R$. For example, $E_{8}(q)=E_{4}^{2}(q)$.
%\end{fact}

\section{Eisenstein series}%:\UTF{2022}%
\label{s2}
\renewcommand{\ep}{\varepsilon}
\renewcommand{\P}{\mathcal{P}}
\renewcommand{\Q}{\mathcal{Q}}
\newcommand{\E}{\mathcal{E}}
\renewcommand{\R}{\mathcal{R}}

We first review some results on $P, Q, R$ originally due to Ramanujan \cite{ra}; see also Berndt \cite[Chapter 14]{be}. 
After that, we will introduce Hahn's Eisenstein-type  series $\P, \E, \Q$, new series $\psi_{n}(q), \ep_{n}(q)$ and then go into our main discussion.

% some related work are Berndt-Chan-Liu-Yesilyurt \cite{bcly}, Chan-Cooper-Toh \cite{cct}. 
 
\subsection{Ramanujan's work}%:\UTF{2022}%
\label{s21}
\begin{defn}
For $n\ge0$, let 
\[
T_{2n}(q)=
1+
\dsum_{k=1}^{\mug}(-1)^{k}
\k{
(6k-1)^{2n}q^{k(3k-1)/2}
+
(6k+1)^{2n}q^{k(3k+1)/2}
}.
\]
\end{defn}
In particular, 
$T_{0}(q)=(q;q)_{\mg}$; 
integers in the form $k(3k\pm 1)$ are \eh{pentagonal numbers}. Apparently, 
the relation $24qT_{2n}'=T_{2n+2}-T_{2n}$ holds 
because 
\[
24\ff{k(3k+1)}{2}=(6k+1)^{2}-1.
\]
\begin{thm}[{\cite[p.357]{be}}]\label{t1}
\begin{align*}
	\ff{T_{2}}{T_{0}}&=P.
	\\\ff{T_{4}}{T_{0}}&=3P^{2}-2Q.
	\\\ff{T_{6}}{T_{0}}&=15P^{3}-30PQ+16R.
	\\\ff{T_{8}}{T_{0}}&=105P^{4}-420P^{2}Q+
	448PR-132Q^{2}.
\end{align*}
%Moreover, $\tf{T_{2n}}{T_{0}}\in \zz[P, Q, R]$.
\end{thm}

There are similar results on other series.

%\subsection{triangular series}%:\UTF{2022}%

\begin{defn}
For $n\ge0$, let 
\[
F_{n}(q)=
\dsum_{k=0}^{\mug}(-1)^{k}(2k+1)^{n+1}q^{k(k+1)/2}.
\]
\end{defn}
In particular, 
\[
F_{0}(q)=
\dsum_{k=0}^{\mug}(-1)^{k}(2k+1)q^{k(k+1)/2}
=\dprod_{n=1}^{\mug} (1-q^{n})^{3}
\]
is the \eh{Gauss-Jacobi identity}. 
Integers in the form $k(k+1)/2$ are \eh{triangular numbers}. 
These series satisfy 
$8qF_{n}'=F_{n+2}-F_{n}$ because 
\begin{align*}
	8\ff{k(k+1)}{2}&=(2k+1)^{2}-1.
\end{align*}

\begin{thm}[{\cite[p.363]{be}}]\label{t2}
\begin{align*}
	\ff{F_{2}}{F_{0}}&=P.
	\\\ff{F_{4}}{F_{0}}&=\ff{1}{3}(5P^{2}-2Q).
	\\\ff{F_{6}}{F_{0}}&=\ff{1}{9}(35P^{3}-42PQ+16R).
	\\\ff{F_{8}}{F_{0}}&=\ff{1}{3}(35P^{4}-84P^{2}Q-12Q^{2}+64PR).
\end{align*}
%Moreover, $\tff{F_{2n}}{F_{0}}\in\qq[P, Q, R]$.
\end{thm}

%\begin{thm}[Ewell]
%\[
%\dsum_{k\ge0}(-1)^{k+1}(2k+1)
%\s(n-\ff{k(k+1)}{2})=
%\begin{cases}
%	0&n\ne \ff{m(m+1)}{2}	\\
%	(-1)^{m}
%	\ff{m(m+1)(2m+1)}{6}
%	&n=\ff{m(m+1)}{2}\\
%\end{cases}
%\]
%\end{thm}

%\np%----------------------------------------------------
%\np%----------------------------------------------------
\subsection{Hahn's Eisenstein series}%:\UTF{2022}%

\renewcommand{\sh}{\widehat{\sigma}}
\renewcommand{\st}{\widetilde{\sigma}}
Following Hahn \cite{ha}, we introduce variants of 
divisor sum functions and Eisenstein-type series.
\begin{defn}
\begin{align*}
	\st_{s}(n)&=\sum_{d|n}(-1)^{d-1}d^{s}.
\\\sh_{s}(n)&=\sum_{d|n}(-1)^{n/d-1}d^{s}.
\end{align*}
%=
%
Often, we write 
$\st_{1}(n)=
\st_{}(n), 
\sh_{1}(n)=
\sh_{}(n)$. 
For convenience, set 
$\s_{s}(n)=\st_{s}{(n)}=\sh_{s}(n)=0$ whenever $n$ is not a nonnegative integer. We will define values of these functions at $n=0$ later on.
\end{defn}
\begin{rmk}\hf
\begin{enumerate}
\item These functions are related to Lambert-type series in the following way.
\begin{align*}
	\dsum_{k=1}^{\mug}\ff{k^{s}q^{k}}{1-q^{k}}
&=\dsum_{n=1}^{\mug}\s_{s}(n)q^{n}.
	\\\dsum_{k=1}^{\mug}\ff{(-1)^{k-1}k^{s}q^{k}}{1-q^{k}}
&=\dsum_{n=1}^{\mug}\st_{s}(n)q^{n}.
	\\\dsum_{k=1}^{\mug}\ff{k^{s}q^{k}}{1+q^{k}}
&=\dsum_{n=1}^{\mug}\sh_{s}(n)q^{n}.
\end{align*}
\item Observe the following properties.
\[
\st_{s}(n)=\s_{s}(n)-2^{s+1}\s_{s}
\k{\ff{n}{2}}.
\]
\[
\sh_{s}(n)=\s_{s}(n)-2\s_{s}\k{\ff{n}{2}}.
\]
\[\st_{s}(n)=\sum_{
\substack{d|n\\d \text{ odd}}
}
d^{s}-
\sum_{
\substack{d|n\\d \text{ even}}
}
d^{s}, \q 
\sh(n)=
	\sum_{\substack{d|n\\d \text{ odd}}}d^{s}.
\]	
In particular, if $n$ is odd, then 
\[
\st_{s}(n)=\s_{s}(n)=\sh_{s}(n).
\]
\end{enumerate}

\end{rmk}

\begin{defn}
Let 
$\st(0)=\tf{1}{8}, \sh(0)=\tf{1}{24}, \st_{3}(0)=-\tf{1}{16}$ and 
\begin{align*}
	\P(q)&=
\dsum_{n=0}^{\mug}8\st(n)q^{n},
	\\\E(q)&=
	\dsum_{n=0}^{\mug}24\sh(n)q^{n},
	\\\Q(q)&=
	\dsum_{n=0}^{\mug}(-16\st_{3}(n))q^{n}.
\end{align*}
\end{defn}
\begin{prop}[{\cite[p.1598, 1595]{ha}}]\label{p1}
Let $x, z, q$ be as in Introduction. Then 
\begin{align*}
	\P&=z^{2}(1-x)+4x(1-x)z\ff{dz}{dx},
	\\\E&=z^{2}(1+x),
	\\\Q&=z^{4}(1-x),
	\\q\P'&=\ff{\P^{2}-\Q}{4},
	\\q\E'&=\ff{\E\P-\Q}{2},
	\\q\Q'&=\P\Q-\E\Q.
\end{align*}
\end{prop}

Now, in addition to Hahn's $\P, \E, \Q$, let us 
introduce 
\[
\R(q)=
\dsum_{n=0}^{\mug}8\st_{5}(n)q^{n}.
\]
Then $\R(q)=\E(q)\Q(q)$ because  
$\E(q)=z^{2}(1+x), \Q(q)=z^{4}(1-x)^{2}$
and  
$\R(q)=z^{6}(1-x)(1-x^{2})$ \cite[p.130 Entry 14 (vi)]{be}.

\subsection{Unsigned series on triangular numbers}

\begin{defn}
\[
\psi_{n}(q)=
\dsum_{k=0}^{\mug}(2k+1)^{n}q^{k(k+1)/2}.
\]
\end{defn}
%Thus, this is the unsigned series with exponents all   numbers.
In particular, $\psi_{0}(q)=\psi(q)$. 
Observe that 
$8q\psi_{n}'=\psi_{n+2}-\psi_{n}$.
%\begin{proof}
%These follow from identities 
%%similar to $T_{2n}$ and $J_{2n+1}$.
%\end{proof}

%\np%----------------------------------------------------

\begin{thm}\label{t3}
\begin{align*}
	\ff{\psi_{2}}{\psi_{0}}&=\P.
	\\\ff{\psi_{4}}{\psi_{0}}&=3\P^{2}-2\Q.
	\\\ff{\psi_{6}}{\psi_{0}}&=15\P^{3}-30\P\Q
	+16\R.
\end{align*}
Moreover, each $n\ge0$, 
\[
\ff{\psi_{2n}}{\psi_{0}}=
\sum_{\substack{i, j, k, l\ge0\\
2i+2j+4k+6l=2n}}K_{ijkl}\P^{i}\E^{j}\Q^{k}\R^{l}
\]
for some integers $K_{ijkl}$.
%$\tf{\psi_{2n}}{\psi_{0}}\in \zz[\P, \E, \Q]_{0}$ where 
%$\zz[\P, \E, \Q]_{0}$ is the set of polynomials in $\P, \E, \Q$ over $\zz$ without the constant term.
\end{thm}

\begin{proof}
$8q\psi_{0}'=\ps_{2}-\psi_{0}$
implies 
\[
\ff{\ps_{2}}{\psi_{0}}=1+8q\ff{\psi_{0}'}{\psi_{0}}=\P.
\]
Next, 
apply the operator 
$8q\tf{d}{dq}$ to $\psi_{2}=\psi_{0}\P$.
Using Proposition \ref{p1}, we see that 
\begin{align*}
	\ps_{4}-\ps_{2}&=
(\ps_{2}-\ps_{0})\P+\psi_{0} 8 \ff{\P^{2}-\Q}{4},
	\\\ps_{4}&=
\ps_{2}+
(\ps_{2}-\ps_{0})\P+\psi_{0} 8 \ff{\P^{2}-\Q}{4}
=3\P^{2}-2\Q.
\end{align*}
Continuing this algorithm, 
$8q\psi_{4}'=8q((3\P^{2}-2\Q)\psi_{0})'$ shows that 
\begin{align*}
	\psi_{6}-\psi_{4}&=3\k{
8\cdot 2\P\ff{\P^{2}-\Q}{4}\psi_{0}+\P^{2}(\psi_{2}-\psi_{0})}
-2\k{
8(\P\Q-\R)\psi_{0}+\Q(\psi_{2}-\psi_{0})}
	\\&=
(15\P^{3}-30\P\Q
	+16\R-3\P^{2}+2\Q)\psi_{0}
\end{align*}
and bring $-\psi_{4}$ to the right hand side. 
We can confirm the last assertion by induction on $2n$. Since this is a routine, we omit details.
%, as we just have seen the cases for $2n=2, 4, 6$. 
%Assume here that 
%\[
%g_{2n}(\P, \E, \Q):=\ff{\psi_{2n}}{\psi_{0}}\in \zz[\P, \E, \Q]_{0}, \q 2n\ge6.
%\]
%Then it follows from $8q(\psi_{2n})'=8q(g_{2n}\psi_{0})'$ 
%that 
%\begin{align*}
%	\psi_{2n+2}&=\psi_{2n}+
%	8qg_{2n}'\psi_{0}+g_{2n}(\psi_{2}-\psi_{0})
%	\\&=(8qg_{2n}'+g_{2n}P)\psi_{0}.
%\end{align*}
%Clearly, $g_{2n+2}(\P, \E, \Q):=
%8qg_{2n}'(\P, \E, \Q)+g_{2n}(\P, \E, \Q)P$
%satisfies 
%$g_{2n+2}\in \zz[\P, \E, \Q]_{0}$ 
%since $8q\P', 8q\E', 8q\Q', g_{2n}P\in \zz[\P, \E, \Q]_{0}$.
\end{proof}
\begin{rmk}
Thus, 
$P\mto \P, 
Q\mto \Q, 
R\mto \R$ gives the correspondence 
between $\tf{T_{2}}{T_{0}}, 
\tf{T_{4}}{T_{0}}, \tf{T_{6}}{T_{0}}$ 
and  $\tf{\psi_{2}}{\psi_{0}}, 
\tf{\psi_{4}}{\psi_{0}}, \tf{\psi_{6}}{\psi_{0}}$;
this is because 
\[
24qP'=2(P^{2}-Q),
8q\P'=2(\P^{2}-\Q),
\]
\[
24qQ'=8(PQ-R),
8q\E'=8(\P\Q-\R)
\]
and $\tf{T_{2}}{T_{0}}=P, \tf{\psi_{2}}{\psi_{0}}=\P.$
However, 
$\tff{\psi_{8}}{\psi_{0}}$ is not quite equal to 
\[
105\P^{4}-420\P^{2}\Q+448\P\R-132\Q^{2}
\]
as $24qR'=12(PR-Q^{2})$ while 
\begin{align*}
	8q\R'&=8q(\E'\Q+\E\Q')
	\\&=4(\P\E-\Q)\Q+8\E(\P\Q-\R)
	\\&=12\P\R-4\Q^{2}-8\E\R
	\ne 12(\P\R-\Q^{2}).
\end{align*}
\end{rmk}

%\np%----------------------------------------------------

%\begin{prop}
%\[
%\ep(q)=f(q, q^{2}).
%\]
%\end{prop}
%\begin{proof}
%
%\end{proof}
%
%\[
%E(q)=
%\]
\subsection{Unsigned series on pentagonal numbers}%:\UTF{2022}%
\begin{defn}For $n\ge0$, let 
\[
\ep_{n}(q)=
\dsum_{k=-\mg}^{\mug}(6k+1)^{n}q^{k(3k+1)/2}.
\]
\end{defn}
Observe again that $24q\ep_{n}'=\ep_{n+2}-\ep_{n}$.
\begin{fact}[{\cite[p.114]{be}}]\label{f1}
\begin{align*}
	\ep_{1}(q)&=\dsum_{k=-\mg}^{\mug}(6k+1)q^{k(3k+1)/2}
=\varphi^{2}(-q)f(-q).
\end{align*}
\end{fact}
As a consequence of this and $\varphi(-q)=(q;q)_{\mg}(q;q^{2})_{\mg}$, we have 
\[
\ep_{1}(q)=
(q;q)_{\mg}^{2}(q;q^{2})_{\mg}^{2}(q;q)_{\mug}
=\dprod_{n=1}^{\mug} (1-q^{n})^{3}(1-q^{2n-1})^{2}.
\]
Hence 
\[
24q\ff{\ep_{1}'}{\ep_{1}}=
-24
{\dsum_{n=1}^{\mug}
(3\s(n)+2\sh(n))q^{n}
}=
3(P-1)-2(\E-1)
.
\]
Now $24q\ep_{1}'=\ep_{3}-\ep_{1}$ further implies that 
\[
\ff{\ep_{3}}{\ep_{1}}=
1+24q\ff{\ep_{1}'}{\ep_{1}}
=3P-2\E.\]
%\begin{prop}
%$\ff{\ep_{3}}{\ep_{1}}=4P-3\P$.
%\end{prop}
%\begin{proof}
%$
%\[
%\ff{\ep_{3}}{\ep_{1}}=
%1+24q\ff{\ep_{1}'}{\ep_{1}}
%=
%1-24
%\k{\dsum_{n=1}^{\mug}
%(4\s(n)-\sh(n))q^{n}
%}
%=1+4(P-1)-3(\P-1)=4P-3\P.
%\]
%\end{proof}
Because of this relation, we expect that 
$\tf{\ep_{2n+1}}{\ep_{1}}$ 
be a polynomial of $P, Q, R$ and $\P, \E, \Q$ altogether.
However, as we see just below, 
each of $P, Q, R$ is indeed a polynomial of 
$\P, \E, \Q$ so that 
$\tf{\ep_{2n+1}}{\ep_{1}}$ is a polynomial in only $\P, \E, \Q$.
\begin{lem}\label{l1}
\begin{align*}
	P&=3\P-2\E.
	\\Q&=4\E^{2}-3\Q.
	\\R&=-8\E^{3}+9\R.
\end{align*}
%Moreover, $R\in\qq[\P, \E, \Q]$.
\end{lem}
\begin{proof}
Let $x, z, q$ be as in Introduction. 
Thanks to Fact \ref{f0} and Proposition \ref{p1}, 
we find 
\begin{align*}
	P-(3\P-2\E)&=
	\k{(1-5x)z^{2}+12x(1-x)z\tf{dz}{dx}}
\\&\ph{=}-3\k{(1-x)z^{2}+4x(1-x)z\tf{dz}{dx}}
+2z^{2}(1+x)=0.
\end{align*}
\begin{align*}
	Q&=P^{2}-12qP'
	\\&=(3\P-2\E)^{2}-12q (3\P'-2\E')
	\\&=(3\P-2\E)^{2}
	-36 \ff{\P^{2}-\Q}{4}+24
\ff{\P\E-\Q}{2}
	\\&=4\E^{2}-3\Q.
\end{align*}
The last equality follows from 
\begin{align*}
	R&=PQ-3qQ'
	\\&=(3\P-2\E)(4\E^{2}-3\Q)
	-3\k{4\cdot 2\E \ff{\P\E-\Q}{2}}-3(-3)(\P\Q-\R)
	\\&=-8\E^{3}+9\R.
\end{align*}
\end{proof}

\begin{thm}\label{t4}
\begin{align*}
	\ff{\ep_{3}}{\ep_{1}}&=9\P-8\E.
	\\\ff{\ep_{5}}{\ep_{1}}&=
	135\P^{2}-240\P\E+64\E^{2}+42\Q.
\end{align*}
Moreover, each $n\ge0$, 
\[
\ff{\ep_{2n+1}}{\ep_{1}}=
\sum_{\substack{i, j, k, l\ge0\\
2i+2j+4k+6l=2n}}K_{ijkl}\P^{i}\E^{j}\Q^{k}\R^{l}
\]
for some integers $K_{ijkl}$.
\end{thm}
\begin{proof}
We readily find that 
\[
\ff{\ep_{3}}{\ep_{1}}=
3P-2\E=
3(3\P-2\E)-2\E=9\P-8\E.
\]
Apply the operator $24q\tf{d}{dq}$ to $\ep_{3}=
(9\P-8\E)\ep_{1}.$
\begin{align*}
	\ep_{5}-\ep_{3}&=
	9\k{
	24\ff{\P^{2}-\Q}{4}\ep_{1}
	+\P(\ep_{3}-\ep_{1})
	}
	\\&\ph{=}-8\k{
	24\ff{\P\E-\Q}{2}
	\ep_{1}+\E(\ep_{3}-\ep_{1})}
	\\&=
	(135\P^{2}-240\P\E+64\E^{2}+42\Q
-9\P+8\E)\ep_{1}.
\end{align*}
We can prove the last statement by induction on $2n+1$ quite similarly.
\end{proof}
In fact, we obtained some unexpected identities with our four kinds of series mixed. Let us record some here (and we can derive more one after another).
\begin{cor}
\begin{align*}
	4\psi_{0}\ep_{1}T_{2}&=T_{0}(3\ep_{1}\psi_{2}+\psi_{0}\ep_{3}).
	\\4\psi_{0}\ep_{1}F_{2}&=F_{0}(3\ep_{1}\psi_{2}+\psi_{0}\ep_{3}).
\end{align*}
\end{cor}
\begin{proof}
\[
\ff{T_{2}}{T_{0}}=
\ff{F_{2}}{F_{0}}=
P=
3\P-2\E=
3\ff{\psi_{2}}{\psi_{0}}-
2\ff{1}{8}\k{9\ff{\psi_{2}}{\psi_{0}}-
\ff{\ep_{3}}{\ep_{1}}}
=
\ff{1}{4}\k{3\ff{\psi_{2}}{\psi_{0}}+\ff{\ep_{3}}{\ep_{1}}}.
\]

\end{proof}

For integers $r, s$ such that $s\ge r\ge0$, $r+s$ is even, Ramanujan considered sums \cite{ra}
\[
\Phi_{r, s}(q)=
\dsum_{m, n=1}^{\mug}m^{r}n^{s}q^{mn}
\,
\k{=\k{q\ff{d}{dq}}^{r}\dsum_{n=1}^{\mug}\s_{s}(n)q^{n}}
.
\]
He recorded, for example, 
\begin{align*}
	288\Phi_{1,2}&=Q-P^{2},
	\\720\Phi_{1,4}&=PQ-R,
	\\1008\Phi_{1,6}&=Q^{2}-PR,
	\\1728\Phi_{2,3}&=3PQ-2R-P^{3},
	\\1728\Phi_{2,5}&=P^{2}Q-2PR+Q^{2},
	\\1728\Phi_{2,7}&=2PQ^{2}-P^{2}R-QR.
\end{align*}
He also claimed that 
such a $\Phi_{r,s}$ is always a polynomial in $P, Q, R$. 
%\hf
Likewise, set 
\[
\widetilde{\Phi}_{r, s}(q)=
\k{q\ff{d}{dq}}^{r}\dsum_{n=1}^{\mug}\st_{s}(n)q^{n}.
%\k{=
%\dsum_{n=1}^{\mug}n^{r}\st_{s}(n)q^{n}}.
\]
Notice that 
\[
{q\ff{d}{dq}}
\widetilde{\Phi}_{r, s}(q)=\widetilde{\Phi}_{r+1, s+1}(q).
\]
\begin{thm}\label{tr}
\begin{align*}
	32\widetilde{\Phi}_{1, 2}&=\P^{2}-\Q.
	\\16\widetilde{\Phi}_{1, 4}&=-\P\Q+\R.
	\\16\widetilde{\Phi}_{1, 6}&=
	3\P\R-\Q^{2}-2\E\R.
	\\64\widetilde{\Phi}_{2, 3}&=
	\P^{3}-3\P\Q+2\R.
	\\64\widetilde{\Phi}_{2, 5}&=
	-5\P^{2}\Q-\Q^{2}+10\P\R-4\E\R.
	\\64\widetilde{\Phi}_{2, 7}&=
21\P^{2}\R+13\Q\R-14\P\Q^{2}-28\P\E\R+8\E^{2}\R.
\end{align*}
Moreover, if $s\ge r\ge 0, r+s$ is even, then 
\[
\widetilde{\Phi}_{r, s}=
\sum_{\substack{i, j, k, l\ge0\\
2i+2j+4k+6l=r+s+1}}
K_{ijkl}\P^{i}\E^{j}\Q^{k}\R^{l}
\]
for some rational numbers $K_{ijkl}$. 
\end{thm}
\begin{proof}
The first three equalities are 
equivalent to the last three ones in Proposition \ref{p1}. 
Applying $q\tf{d}{dq}$ to those again, 
we have:
\begin{align*}
	\widetilde{\Phi}_{2, 3}&=
	q\ff{d}{dq}
\k{\widetilde{\Phi}_{1, 2}}
	\\&=
\ff{1}{32}
\k{2\P \ff{\P^{2}-\Q}{4}-(\P\Q-\R)}.
	\\\widetilde{\Phi}_{2, 5}&=q\ff{d}{dq}
\k{\widetilde{\Phi}_{1, 4}}
	\\&=
-\ff{1}{16}
\k{\ff{\P^{2}-\Q}{4}\Q+\P(\P\Q-\R)
-\ff{1}{8}(12\P\R-4\Q^{2}-8\E\R)
}.
	\\\widetilde{\Phi}_{2, 7}&=q\ff{d}{dq}
\k{\widetilde{\Phi}_{1, 6}}
	\\&=\ff{1}{16}
\left(3\ff{\P^{2}-\Q}{4}\R+
+3\P\ff{1}{8}(12\P\R-4\Q^{2}-8\E\R)
\right.
\\&\left.
-2\Q(\P\Q-\R)-2\ff{\P\E-\Q}{2}\R
-2\E\ff{1}{8}(12\P\R-4\Q^{2}-8\E\R)\right).
\end{align*}
With little more manipulation, we will arrive at the desired results.
\end{proof}

\begin{cor}\label{t5}
\begin{align*}
	\widetilde{\Phi}_{0, 1}(-e^{-\p})&=\ff{1}{4\p}-\ff{1}{8},
	\\\widetilde{\Phi}_{0, 5}(-e^{-\p})&=-\ff{1}{8},
\\
\widetilde{\Phi}_{1, 6}(-e^{-\p})&=
-\ff{\p^{4}}{16\,\Gamma^{16}\k{\tf{3}{4}}},
\\\widetilde{\Phi}_{2, 7}(-e^{-\p})&
=-\ff{7\p^{3}}{16\,\Gamma^{16}\k{\tf{3}{4}}}.
%\sum_{n=1}^{\mg}n \,\st_{5}(n)(-e^{-\p})^{n}
%	\\&=\f{1}{8}q\R'(q)\Bigr|_{q=-e^{-\p}}
%	\\&=\ff{-e^{-\p}}{8}
%	\k{\E'(-e^{-\p})\Q(-e^{-\p})+\underbrace{\E(-e^{-\p})}_{0}
%\Q'(-e^{-\p})}\\&
\end{align*}
\end{cor}

\begin{proof}
First, we need to mention that 
the all of following hold.
\[
\P(-e^{-\p})=\ff{2}{\p}, \q
\Q(-e^{-\p})=\ff{\p^{2}}{\Gamma^{8}(\tf{3}{4})}, \q
\E(-e^{-\p})=\R(-e^{-\p})=0.
\]
In fact, Berndt \cite[p.293]{be} shows $\E(-e^{-\p})=0$ ($B(q)$ in \cite[Chapter 11]{be} is identical to $\E(q)$).
For all the others, recall the classical evaluation 
\[
P(-e^{-\p})=2P(e^{-2\p})=\ff{6}{\p},\q
Q(-e^{-\p})=-4Q(e^{-2\p})=\ff{3\p^{2}}{4\Gamma^{8}(\tf{3}{4})},
\]
$R(-e^{-\p})=0$ \cite[p.398, p.306]{be}
and Lemma \ref{l1}. 
With these values, it is now easy to see 
\begin{align*}
	\widetilde{\Phi}_{0, 1}(-e^{-\p})&=\ff{1}{8}(\P(-e^{-\p})-1)=\ff{1}{4\p}-\ff{1}{8},
	\\\widetilde{\Phi}_{0, 5}(-e^{-\p})&=
\ff{1}{8}(\R(-e^{-\p})-1)=-\ff{1}{8}.
\end{align*}
For the other two equalities, use Theorem \ref{tr} and 
$\R(-e^{-\p})=0$.
\end{proof}
\begin{cor}
\[
\widetilde{\Phi}_{0, 9}(-e^{-\p})=
-\ff{31}{8}, \q 
\widetilde{\Phi}_{0, 13}(-e^{-\p})=
-\ff{5461}{8}. 
\]
\end{cor}
\begin{proof}
\begin{align*}
	\widetilde{\Phi}_{0, 9}(q)&=
	\dsum_{n=1}^{\mug}\st_{9}(n)q^{n}
	\\&=
	\dsum_{n=1}^{\mug}
	\k{\s_{9}(n)-2^{10}\s_{9}\k{\ff{n}{2}}}
q^{n}
	\\&={\Phi}_{0, 9}(q)-1024{\Phi}_{0, 9}(q^{2}).
\end{align*}
Note that 
$E_{10}(-e^{-\p})=
E_{10}(e^{-2\p})=0$ 
since 
$E_{10}(q)=Q(q)R(q)$ and $R(-e^{-\p})=R(e^{-2\p})=0$.
Therefore, 
\[
\widetilde{\Phi}_{0, 9}(-e^{-\p})=
\ff{1}{-264}(E_{10}(-e^{-\p})-1)
-1024\ff{1}{-264}(E_{10}(e^{-2\p})-1)
=-\ff{31}{8}.
\]
Using the fact $E_{14}(q)=Q^{2}(q)R(q)$, 
we similarly have 
\[
\widetilde{\Phi}_{0, 13}(-e^{-\p})=
\ff{1}{24}(E_{14}(-e^{-\p})-1)
-16384\ff{1}{24}(E_{14}(e^{-2\p})-1)
=-\ff{5461}{8}.
\]

\end{proof}

%Therefore, it is possible to evaluate the higher derivatives 
%$(q\tff{d}{dq})^{r}
%\P(q)\bigr|_{q=-e^{-\p}}$ etc. by induction on $r$.
%These give certain sums such as 
%\[
%\widetilde{\Phi}_{r, s}(q)=
%\sum_{n=1}^{\mg}n^{r}\,\st_{s}(n)q^{n}.
%\]
%With the help of all the discussions above, we can compute some values such as 
%\item Thus, we can also find 
%$\psi_{2}(-e^{-\p}), \psi_{4}(-e^{-\p}), \psi_{6}(-e^{-\p})$  
%since Yi-Lee-Paek showed $\psi(-e^{-\p})=2^{-3/4}e^{\p/8}\tf{\p^{1/4}}{\Gamma\k{\tf{3}{4}}}$ \cite[p.175]{ylp}. 
%\end{enumerate}

%\np%----------------------------------------------------

%\np%----------------------------------------------------

%\[
%.
%\]

%\begin{rmk}
%This value is indeed
%\[
%$l_{91}', 
%l_{92}', 
%l_{93}', 
%l_{94}',
%l_{95} $

%\np%----------------------------------------------------
\section{Convolution sums of divisors}%:\UTF{2022}%
\label{s3}

%Ramanujan and many other researchers found 
%convolution sums for divisor sum functions. 
The aim of this section is to observe several variants of  convolution sums of divisors with examples.  We start with the simple one.
%For convenience, introduce 

%\begin{lem}
%Moreover, 
%\[
%q\R'=
%\P\Q^{2}-\Q\R+\ff{1}{2}(\E^{2}\P-\R).
%\]
%\end{lem}

%\[
%q\R'=q(\E\Q)'
%=(\P\Q-\R)\Q+\E\ff{\E\P-\Q}{2}.\]
%\end{proof}
\begin{thm}\label{t10}
For each $n\ge0$, we have 
\[
\st_{5}(n)=-48
\dsum_{
\substack{i, j\ge0\\
i+j=n}
}\sh(i)\st_{3}(j)
\]
with 
$\st_{5}(0)=\tf{1}{8}, \sh(0)=\tf{1}{24}$ and 
$\st_{3}(0)=-\tf{1}{16}$.
\end{thm}

\begin{proof}
The identity $\R(q)=\E(q)\Q(q)$ is equivalent to 
\[
\dsum_{n=0}^{\mug}8\st_{5}(n)q^{n}=
\k{\dsum_{i=0}^{\mug}24\sh(i)q^{i}}
\k{\dsum_{j=0}^{\mug}(-16\st_{3}(j))q^{j}}
.
\]
Equate coefficients of $q^{n}$ of these series.
\end{proof}

\begin{ex}
Observe $\st_{5}{(4)}=1^{5}-2^{5}-4^{5}=-1055$
 while 
\[
-48\k{
\sh(0)\st_{3}(4)
+
\sh(1)\st_{3}(3)
+
\sh(2)\st_{3}(2)
+
\sh(3)\st_{3}(1)+
\sh(4)\st_{3}(0)
}
\]
\[=
-48\k{\ff{-71}{24}+1\cdot 28
+1(-7)+4\cdot 1+1
\k{-\ff{1}{16}}
}=-1055.\]
\end{ex}

In Lemma \ref{l1}, we showed that $P=3\P-2\E$. This is equivalent to 
\begin{align*}
	-24\s(n)&=3\cdot 8\st(n)-2\cdot 24\sh(n),
	\\\s(n)+\st(n)&=2\sh(n),
\end{align*}
an almost trivial identity. 
However, the other two identities on $Q, R$ give 
arithmetical identities which are not so trivial.
One of them involves even a \eh{three-term} convolution  sum.
\begin{thm}\label{t11}
For each $n\ge0$, we have 
\begin{align*}
	5\s_{3}(n)-\st_{3}(n)&=48
\dsum_{
\substack{i, j\ge0\\
i+j=n}
}\sh(i)\sh(j),
	\\7\s_{5}(n)+\st_{5}(n)&=1536
\dsum_{
\substack{i, j, k\ge0\\
i+j+k=n}
}\sh(i)\sh(j)\sh(k)
\end{align*}
with 
$\s_{3}(0)=\tf{1}{240}, 
\st_{3}(0)=-\tf{1}{16}, 
\sh(0)=\tf{1}{24}, 
\s_{5}(0)=-\tf{1}{504}$ and 
$\st_{5}(0)=\tf{1}{8}.$
\end{thm}

\begin{proof}
The first equality is a consequence of 
$Q=4\E^{2}-3\Q$, that is, 
\[
240\s_{3}(n)=4\cdot 24^{2}
\dsum_{
\substack{i, j\ge0\\
i+j=n}
}\sh(i)\sh(j)-3(-16)\st_{3}(n).
\]
The second one is equivalent to $R=-8\E^{3}+9\R.$ 
The idea is now clear. So we omit a proof.
\end{proof}

If $N\in \nn$, say $N=2^{k}n$ with $k\ge0$, $n$ \text{odd}, 
then by multiplicity of $\s_{s}$ we have 
\[
\s_{s}(N)=\s_{s}(2^{m}n)=
\s_{s}(2^{m})\s_{s}(n)
=(2^{m+1}-1)\s_{s}(n).
\]
Thus, it is crucial to determine values 
$\s_{s}(n)$ for $n$ \eh{odd}. 
Here we show two simple formulas 
with two kinds of proofs.

\renewcommand{\sodd}{\sh}

\begin{cor}\label{cor1}
For $n$ odd, we have 
\begin{align*}
	\s_{3}(n)&=12
\dsum_{
\substack{i, j\ge0\\
i+j=n}
}\sodd(i)\sodd(j)\q (\text{cf. Theorem \ref{t9}}),
	\\\s_{5}(n)&=192
\dsum_{
\substack{i, j, k\ge0\\
i+j+k=n}
}\sodd(i)\sodd(j)\sodd(k)
\end{align*}
with $\sh(0)=\tf{1}{24}$.
\end{cor}

\begin{proof}[First proof]
Recall that $\st_{3}(n)=\s_{3}(n)$ 
and $\st_{5}(n)=\s_{5}(n)$ for $n$ odd. So use Theorem \ref{t11}.
\end{proof}
As a matter of fact, the main idea of the second proof frequently appears in the literature. However, we include  it here since such a proof suggests generalization to Eisenstein series of higher degree; see Theorem \ref{t12} below.

\begin{proof}[Second proof]
Recall from \cite[p.126]{be} that 
\begin{align*}
Q(q^{2})&=z^{4}(1-x+x^{2}),
\\R(q^{2})&=z^{6}(1+x)(1-\tf{1}{2}x)(1-2x)
\end{align*}
where $x, z, q$ are given as in Introduction. On the one hand, we have 
\[
4Q(q^{2})+
Q(q)=
4z^{4}(1-x+x^{2})+z^{4}(1+14x+x^{2})
=
5z^{4}(1+x)^{2}=
5\E^{2}(q).
\]
Again, equating coefficients of $q^{n}$ of these for $n$ odd, we have 
\begin{align*}
	240\s_{3}(n)&=5
\dsum_{
\substack{i, j\ge0\\
i+j=n}
}24^{2}\sh(i)\sh(j)
\end{align*}
and compute these constants. On the other hand, we have 
\begin{align*}
 8R(q^{2})-R(q)&=8z^{6}(1+x)(1-\tf{1}{2}x)(1-2x)-
z^{6}(1+x)(1-34x+x^{2})
 	\\&=7z^{6}(1+x)^{3}=7\E^{3}(q)
 \end{align*}
so that we conclude
\[
-(-504\s_{5}(n))=
7\dsum_{
\substack{i, j, k\ge0\\
i+j+k=n}
}24^{3}\sodd(i)\sodd(j)\sodd(k)
\]
for $n$ odd.
%the proof is complete.
\end{proof}
\begin{rmk}
Hahn also proved the first formula implicitly \cite[Theorem 4.2]{ha}.
\end{rmk}
%\np%----------------------------------------------------

\begin{ex}
Observe 
$\s_{3}(5)=1^{3}+5^{3}=126$ while 
\begin{align*}
	12\dsum_{
\substack{i, j\ge0\\
i+j=5}
}\sh(i)\sh(j)&={12}(2\sodd(0)\sodd(5)
+2\sodd(1)\sodd(4)+2\sodd(2)\sodd(3)
)
	\\&=24
\k{\ff{1}{24}6+1\cdot 1+1\cdot 4}
=126.
\end{align*}
\end{ex}

%Yet there is another formula including a \eh{three-term} convolution sum of divisor sum functions.

\begin{ex}
$\s_{5}(3)=1^{5}+3^{5}=244.$
\begin{align*}
	192\dsum_{
\substack{i, j, k\ge0\\
i+j+k=3}
}\sodd(i)\sodd(j)\sodd(k)
&=
192(3\sodd(3)\sodd(0)^{2}
+6\sodd(2)\sodd(1)\sodd(0)
+\sodd(1)^{3}
)
	\\&=192
\k{3\cdot 4 \ff{1}{24^{2}}
+6\cdot 1\cdot 1\cdot \ff{1}{24}+1^{3}}
=244.
\end{align*}
\end{ex}

\begin{thm}\label{t12}
For $n$ odd, we have 
\[
\s_{7}(n)=
12
\k{864
\dsum_{
\substack{i, j, k, l\ge0\\
i+j+k+l=n}
}\sodd(i)\sodd(j)\sodd(k)\sodd(l)
-
\dsum_{
\substack{i, j\ge0\\
i+j=n}
}\sodd(i)\st_{5}(j)
}
\]
with 
$\sh(0)=\tf{1}{24}$ and $\st_{5}(0)=\tf{1}{8}$.
\end{thm}
\begin{proof}
Using 
$5\E^{2}(q)=4Q(q^{2})+Q(q)$, $Q^{2}(q)=E_{4}^{2}(q)=E_{8}(q)$, $Q(q)=4\E^{2}(q)-3\Q(q)$ and 
$\R(q)=\E(q)\Q(q)$ altogether, we find  
\begin{align*}
	25\E^{4}(q)&=16Q^{2}(q^{2})+8Q(q^{2})Q(q)+Q^{2}(q)
	\\&=16Q^{2}(q^{2})+2Q(q)(
	5\E^{2}(q)-Q(q)
	)+Q^{2}(q)
	\\&=16E_{8}(q^{2})+10\E^{2}(q)Q(q)-E_{8}(q)
\end{align*}
and hence 
\begin{align*}
	E_{8}(q)-16E_{8}(q^{2})&=10\E^{2}(q)(4\E^{2}(q)-3\Q(q))-25\E^{4}(q)
	\\&=15\E^{4}(q)-30\E(q)\R(q).
\end{align*}
Equate coefficients of $q^{n}$, $n$ odd.
\end{proof}

\begin{ex}
$\s_{7}(3)=1^{7}+3^{7}=2188.$
%\[
%\dsum_{
%\substack{i, j, k\ge0\\
%i+j+k=3}
%}\sodd(i)\sodd(j)\s_{3}(k)
%\]
%\[=
%2\sh(3)\sh(0)
%\s_{3}(0)
%+\sh(0)^{2}\s_{3}(3)+
%2\sh(2)\sh(1)\s_{3}(0)
%+2\sh(2)\sh(0)\s_{3}(1)
%+2\sh(1)\sh(0)\s_{3}(2)
%+\sh(1)^{2}\s_{3}(1)
%\]
%\[
%=
%2\cdot 4\ff{1}{24}\ff{1}{240}
%+
%\ff{28}{24^{2}}+\ff{2}{240}+
%\ff{2}{24}+\ff{2\cdot 9}{24}+1=\ff{227}{120}
%\]
Let us compute the two sums above for $n=3$ 
separately.
\begin{align*}
	\dsum_{
\substack{i, j, k, l\ge0\\
i+j+k+l=3}
}\sodd(i)\sodd(j)\sodd(k)\sodd(l)
&=
4\sh(3)\sh(0)^{3}
+12\sh(2)\sh(1)\sh(0)^{2}+
4\sh(1)^{3}\sh(0)
	\\&=
4\cdot 4\ff{1}{24^{3}}+12\ff{1}{24^{2}}
+4\ff{1}{24}=\ff{163}{864}.
\end{align*}
\[
\dsum_{
\substack{i, j\ge0\\
i+j=3}
}\sodd(i)\st_{5}(j)
=
\sodd(0)\st_{5}(3)
+\sodd(1)\st_{5}(2)
+\sodd(2)\st_{5}(1)
+\sodd(3)\st_{5}(0)
\]
\[=
\ff{244}{24}+1(-31)+1^{2}+\ff{4}{8}=-\ff{58}{3}.
\]
Therefore, 
\[
12
\k{864\,\ff{163}{864}-\k{-\ff{58}{3}}
}
=2188.
\]
\end{ex}

%Addition to these convolution formulas, 
%we can find some relations of 
%several divisor sum functions.

%\begin{cor}
%For $n$ odd,
%\[
%\s_{5}(n)-2\s_{3}(n)+\s(n)=
%192
%\dsum_{
%\substack{i, j, k>0\\
%i+j+k=n}
%}\sodd(i)\sodd(j)\sodd(k).
%\]
%\end{cor}
%\begin{proof}
%Let us split the sum 
%$\sum_{
%\substack{i, j, k\ge 0\\
%i+j+k=n}
%}\sodd(i)\sodd(j)\sodd(k)$ into the following three parts:
%\[
%S_{1}(n)=3
%\dsum_{
%\substack{i=n, j=k=0
%}
%}\sodd(i)\sodd(j)\sodd(k)=\ff{1}{192}\sh(n),
%\]
%\[
%S_{2}(n)=3
%\dsum_{
%\substack{i, j>0, k=0\\
%i+j=n}
%}\sodd(i)\sodd(j)\sodd(k)=
%\ff{1}{8}
%\dsum_{
%\substack{i, j>0\\i+j=n}
%}\sodd(i)\sodd(j)
%=
%\ff{1}{8}\k{\ff{1}{12}(\s_{3}(n)-\sh(n))}
%\]
%\[
%S_{3}(n)=
%\dsum_{
%\substack{i, j, k>0\\i+j+k=n}
%}\sodd(i)\sodd(j)\sodd(k).
%\]
%Thus 
%\begin{align*}
%	\s_{5}(n)&=
%	192(S_{1}(n)+S_{2}(n)+S_{3}(n))
%	\\&=192\k{
%	\ff{1}{192}\sh(n)+
%\ff{1}{8}\,{\ff{1}{12}(\s_{3}(n)-\sh(n))}
%	+S_{3}(n)}.
%\end{align*}
%Manipulate this using $\sh(n)=\s(n)$ ($n$ odd).
%\end{proof}

%\np%----------------------------------------------------

\section{Evaluation of theta functions}%:\UTF{2022}%
\label{s4}

In this section, 
we will determine the special values of 
$\ep_{0}(\pm e^{-\p}), 
\ep_{0}(e^{2\p})$ and 
$f(e^{-\p}, e^{-5\p})$. 
Let us write $\ep_{0}(q)=\ep(q)$, $\psi_{0}(q)=\psi(q)$ 
for simplicity.
%$\ep_{0}(q)= \dsum_{n=-\mug}^{\mug}q^{e_{n}}$. 

\begin{prop}\label{p2}
\[
\psi(q)=\ep(q^{3})+q\psi(q^{9}).
\]
\end{prop}

\begin{proof}
\begin{align*}
	\psi(q)&=
	\dsum_{k=0}^{\mug}q^{k(k+1)/2}
	\\&=
	\dsum_{j=0}^{\mug}q^{3j(3j+1)/2}
	+\dsum_{j=0}^{\mug}q^{(3j+1)(3j+2)/2}
	+\dsum_{j=0}^{\mug}q^{(3j+2)(3j+3)/2}
	\\&=
	\dsum_{j=0}^{\mug}q^{3j(3j+1)/2}
	+\dsum_{j=0}^{\mug}q^{(3j+2)(3j+3)/2}	
	+\dsum_{j=0}^{\mug}q^{(3j+1)(3j+2)/2}
	\\&=
	\k{
	\dsum_{j=0}^{\mug}q^{3j(3j+1)/2}
	+
	\dsum_{j\le-1}^{}q^{3j(3j+1)/2}
	}
+
q
\dsum_{j=0}^{\mug}q^{9j(j+1)/2}
	\\&=\ep(q^{3})+q\psi(q^{9}).
\end{align*}

\end{proof}
\begin{rmk}
Replacing $q$ by $-q$, we also proved 
\[
\psi(-q)=\ep(-q^{3})-q\psi(-q^{9}).
\]
\end{rmk}
Following Yi-Lee-Paek \cite{ylp}, we introduce the following symbols.

\begin{defn}
For $k, n>0$, let 
\begin{align*}
	l_{k, n}=\ff{\psi(-q)}{k^{1/4}q^{(k-1)/8}\psi(-q^{k})}, 
	\q 
l_{k, n}'=
\ff{\psi(q)}{k^{1/4}q^{(k-1)/8}\psi(q^{k})}
\end{align*}
where $q=e^{-\pi\rt{n/k}}$.
\end{defn}

\begin{lem}[{\cite[Theorem 2.3 (i), 3.3 (vii), 5.7 (iii), 5.6 (iii)]{ylp}}]\label{l2}
\begin{align*}
	l_{9,1}&=1, \q 
	l_{9, 1}'=\ff{1}{2}(1+\rt{2}\sqrt[4]{3}+\rt{3}).
	\\\psi(e^{-3\p})&=
	2^{-1/8}3^{-3/8}e^{3\p/8}
	(1+\rt{2}\sqrt[4]{3}+\rt{3})^{-1/2}
	\ff{\pi^{1/4}}{\Gamma(\tf{3}{4})}.
	\\\psi(-e^{-3\p})&=
	2^{-3/4}3^{-1/2}e^{3\p/8}
	(2\rt{3}-3)^{1/4}
	\ff{\pi^{1/4}}{\Gamma(\tf{3}{4})}.
\end{align*}
\end{lem}

%$l_{91}', 
%l_{92}', 
%l_{93}', 
%l_{94}',
%l_{95} $
 
%\begin{lem}
%\[
%\psi(e^{-\p})=
%\]
%\end{lem}

\begin{thm}\label{t6}
\[
\ep(e^{-\p})=
2^{-9/8}3^{-3/8}e^{\p/24}
\ff{1+\rt{3}+\rt{2}\cdot 3^{3/4}}{
(1+\rt{3}+\rt{2}\sqrt[4]{3})^{1/2}}
\ff{\pi^{1/4}}{\Gamma(\tf{3}{4})}.
\]
\[
\ep(-e^{-\p})=
	2^{-3/4}3^{-1/2}e^{\p/24}
	(1+\rt{3})(2\rt{3}-3)^{1/4}
\ff{\pi^{1/4}}{\Gamma(\tf{3}{4})}.
\]
\end{thm}

\begin{proof}
Let $k=9, n=1$ and $q=e^{-\p/3}$. 
Proposition \ref{p2} and Lemma \ref{l2} show that 
\begin{align*}
	\ep(q^{3})&=
	\psi(q)-q\psi(q^{9})
	\\&=
	(l_{9,1}'9^{1/4}-1)q\psi(q^{9})
	\\&=
	\k{
	\ff{1}{2}(1+\rt{2}\sqrt[4]{3}+\rt{3})\rt{3}-1
	}e^{-\p/3}\psi(e^{-3\p})
	\\&=
2^{-9/8}3^{-3/8}e^{\p/24}
\ff{1+\rt{3}+\rt{2}\cdot 3^{3/4}}{
(1+\rt{3}+\rt{2}\sqrt[4]{3})^{1/2}}
\ff{\pi^{1/4}}{\Gamma(\tf{3}{4})}
.
\end{align*}.
\begin{align*}
	\ep(-q^{3})&=
	\psi(-q)+q\psi(-q^{9})
	\\&=
	(l_{9,1}9^{1/4}+1)q\psi(-q^{9})
	\\&=(1+\rt{3})e^{-\p/3}\psi(-e^{-3\p})
	\\&=2^{-3/4}3^{-1/2}e^{\p/24}
	(1+\rt{3})(2\rt{3}-3)^{1/4}
\ff{\pi^{1/4}}{\Gamma(\tf{3}{4})}
.
\end{align*}
\end{proof}

\begin{lem}[{\cite[p.772, Theorem 4.3 (ii)]{yck}, 
\cite[p.175, Theorem 5.6 (vi)]{ylp}}]\label{l3}
\begin{align*}
	l_{9,4}'&=
	1+\rt{3}+\ff{1}{2}(\rt{2}+\rt{6})\sqrt[4]{3}.
	\\\psi(e^{-6\p})&=
	\ff{e^{3\p/4}}{
	3^{3/8}(\rt{3}+1)^{5/6}
	(1+\rt{3}+\rt{2}\cdot 3^{3/4})^{2/3}
	}
	\ff{\p^{1/4}}{\Gamma{(\tf{3}{4})}}.
\end{align*}

\end{lem}
%\begin{rmk}
%%We can evaluate 
%%$\ep(\pm e^{-3\p})$, 
%with values of 
%$l_{94}, l_{94}', l_{99}, l_{99}', \psi(\pm e^{-6\p}), 
%\psi(\pm e^{-9\p})$. 
%\end{rmk}

\begin{thm}\label{t7}
\[
\ep(e^{-2\p})=
\ff{e^{\p/12}
(2+\rt{3}+\tf{1}{2}(\rt{2}+\rt{6})3^{3/4})
}{
	3^{3/8}(\rt{3}+1)^{5/6}
	(1+\rt{3}+\rt{2}\cdot 3^{3/4})^{2/3}
	}
\ff{\p^{1/4}}{\Gamma{(\tf{3}{4})}}.
\]
\end{thm}
\begin{proof}
Let $k=9, n=4$ and $q=e^{-2\p/3}$. 
Then 
$\ep(q^{3})=(l_{9, 4}'-1)q\psi(q^{9})$ 
gives the desired equality.
\end{proof}

%\np%----------------------------------------------------

\begin{lem}\label{lx}
\[
f(q, q^{5})=
\ff{\ep(q)}{\ep(q^{2})}\psi(q^{3}).
\]
\end{lem}
\begin{proof}
Recall that 
\[
f(a, b)=
\dsum_{j=-\mug}^{\mug}a^{j(j+1)/2}b^{(j-1)j/2}.
\]
%Now 
%\[
%f(q, q^{5})=
%\dsum_{j=-\mug}^{\mug}q^{j(3j-2)}
%=
%\dsum_{k=-\mug}^{\mug}q^{k(3k+2)}
%\]
%is the series on the left hand side.
%so that $f(e^{-\p}, e^{-5\p})$ is the series on the left hand side.
%enough to evaluate 
%$f(e^{-\p}, e^{-5\p}).$
Notice that 
\[
f(q, q^{2})=
\dsum_{j=-\mug}^{\mug}q^{j(3j-1)/2}=
\dsum_{k=-\mug}^{\mug}q^{k(3k+1)/2}=
\ep(q).
\]
The identity 
$f(a, ab^{2})f(b, a^{2}b)=f(a, b)\psi(ab)$ 
\cite[p.46]{be} with 
$a=q, b=q^{2}$ yields 
\[
f(q, q^{5})=\ff{f(q, q^{2})}{f(q^{2}, q^{4})}\psi(q^{3})=
\ff{\ep(q)}{\ep(q^{2})}\psi(q^{3}).
\]
\end{proof}

\begin{thm}\label{t8}
$f(e^{-\p}, e^{-5\p})$ is equal to 
\begin{align*}
2^{-5/4}3^{-3/8}e^{\pi/3}
(\rt{3}+1)^{5/6}
\ff{(1+\rt{3}+\rt{2}\cdot 3^{3/4})^{5/3}}{
(1+\rt{3}+\rt{2} \sqrt[4]{3})
(2+\rt{3}+\tf{1}{2}(\rt{2}+\rt{6})3^{3/4})
}
\ff{\p^{1/4}}{\Gamma{(\tf{3}{4})}}.
%	\\&=1.043210689\cd.
\end{align*}
\end{thm}
\begin{proof}
Apply Lemma \ref{lx} with $q=e^{-\p}$ 
and use the values in Lemma \ref{l2}, Theorems \ref{t6}, \ref{t7}.
\end{proof}

%\begin{rmk}
%In this way, we evaluated a little unfamiliar sum 
%\[
%f(e^{-\p}, e^{-5\p})=
%\sum_{k=-\mug}^{\mug}
%(e^{-\p})^{k(3k+2)}.
%\]
%According to Wolfram alpha, this is approximately $1.043210689\cd.$
%%$\ff{\ep(e^{-\p})}{\ep(e^{-2\p})}\psi(e^{-3\p})$
%\end{rmk}

\section{Acknowledgment.}

The author would like to thank my family for encouraging  his work.

\np%----------------------------------------------------

%:%:\UTF{2022}%
\end{document}